\documentclass[10pt]{article}
\usepackage{mathtools}
\usepackage{graphicx}
\usepackage{tikz}
\usepackage{float}
\usetikzlibrary{decorations.markings}
\usepackage{caption}
\usepackage{amsmath,amsfonts,latexsym,amscd,amssymb,theorem}
\usepackage{graphicx}
\usepackage[utf8]{inputenc}
\usepackage{epsfig}
\usepackage{amsmath}
\usepackage{amssymb}
\usepackage{afterpage}
\usepackage{selinput}
\usepackage{hyperref}
\newcommand{\ba}{\begin{eqnarray}}
\newcommand{\ea}{\end{eqnarray}}

\newcommand{\qed}{\hfill\rule{1.7ex}{1.7ex}}

\newtheorem{thm}{Theorem}[section]

\newtheorem{theorem}[thm]{Theorem}

\newtheorem{lemma}[thm]{Lemma}

\makeatletter
\newcommand*{\rom}[1]{\expandafter\@slowromancap\romannumeral #1@}
\makeatother

\vspace{13mm}

\author{
Maidoun Mortada\footnote{KALMA Laboratory, Faculty of Sciences, Lebanese University, Baalbek, Lebanon; 
Graph Theory and Operation Research, Department of Mathematics and Physics, Lebanese International University (LIU), Beirut, Lebanon;
and Basic and Applied Sciences Research, Al Maaref University, Beirut, Lebanon.}
\footnote{Both authors contributed equally to this work.}
\and
Ayman El Zein\footnote{Computer Science Department, University of Sciences and Arts in Lebanon, Beirut, Lebanon; 
KALMA Laboratory, Faculty of Sciences, Lebanese University, Beirut, Lebanon.}
\footnotemark[2]
}
\begin{document} 
\title{A Short Proof that Every Claw-Free Cubic Graph is $(1,1,2,2)$-Packing Colorable}
\maketitle
\begin{abstract}
It was recently proved that every claw-free cubic graph admits a \((1,1,2,2)\)-packing coloring—that is, its vertex set can be partitioned into two 1-packings and two 2-packings. This result was established by Brešar, Kuenzel, and Rall [\textit{Discrete Mathematics} 348 (8) (2025), 114477]. In this paper, we provide a simpler and shorter proof.
\end{abstract}


\section{Introduction}
All graphs considered in this paper are finite and simple. A graph $G$ is said to be \emph{cubic} if $d_G(x) = 3$ for every vertex $x \in V(G)$. A claw is the complete bipartite graph $K_{1,3}$. A graph is said to be claw-free if it contains no induced subgraph isomorphic to a claw. In particular, if a graph is cubic and claw-free, then every vertex belongs to a triangle.

Given a non-decreasing sequence $S = (a_1, a_2, \ldots, a_r)$ of positive integers, an \emph{$S$-packing coloring} of a graph $G$ is a mapping $f : V(G) \to [r]$ such that for every pair of vertices $u, v$ with $f(u) = f(v) = i$, the distance between $u$ and $v$ in $G$ is strictly greater than $a_i$. When $S = (1,2,\ldots,r)$, the coloring is referred to as a \emph{packing coloring}, and the minimum such $r$ is called the \emph{packing chromatic number} of $G$, $\chi_{\rho}(G)$. This parameter was initially introduced by Goddard et al.~\cite{d}. Since then, it has received significant attention, particularly in the context of (sub)cubic graphs (see \cite{ bkl, 3, bf, 11, 8, 9, AM1, AM2, 16, 24, m, LW, LZZ24, AM3,26, 27, mai, mai1, i, a}). A comprehensive overview of developments in this area can be found in the survey by Bre\v{s}ar, \emph{et al.}~\cite{bfk}.

Recently, Brešar, Kuenzel, and Rall~\cite{BKR} established that every claw-free cubic graph admits a $(1,1,2,2)$-packing coloring. This confirms that the conjecture posed by Brešar et al.~\cite{9}-stating that the packing chromatic number of subdivisions of subcubic graphs is at most 5—holds for the class of claw-free cubic graphs. 
In this paper, we present a new and much shorter proof of this result. Our approach is purely combinatorial and avoids any structural decomposition. The key ingredient is a general lemma asserting that every cubic graph admits two disjoint 2-packings whose removal eliminates all triangles. This lemma is of independent interest, as it can be applied to various subclasses of cubic graphs. In the case of claw-free cubic graphs, the absence of triangles after applying the lemma allows a simple completion of the coloring by a $(1,1)$-partition of the remaining vertices, thereby yielding the desired $(1,1,2,2)$-packing coloring in a direct and transparent way.

\section{The Key Lemma: Breaking All Triangles with Two 2-Packings}
The cornerstone of our new proof is a simple yet powerful structural observation about cubic graphs. We show that it is always possible to select two disjoint 2-packings whose removal eliminates every triangle in the graph. This lemma not only serves as the foundation of our main result on claw-free cubic graphs but also stands as an independent tool that may find applications in other packing-coloring contexts. 
\begin{lemma}\label{keylemma}
    Let $G$ be a cubic graph. Then there exist two disjoint 2-packings in $G$, say $X$ and $Y$, such that $G[V(G)\setminus (X\cup Y)]$ is triangle-free.
\end{lemma}
\begin{proof}
  For each vertex $x \in V(G)$, define its weight $w(x)$ by
\[w(x) =
\begin{cases}
k_1 & \text{if $x$ is contained in two triangles,}\\[2mm]
k_2 & \text{if $x$ is contained in exactly one triangle,}\\[2mm]
0 & \text{if $x$ is not contained in any triangle,}
\end{cases}
\]

with $k_1$ and  $k_2$ are nonzero real numbers satisfying $k_1>k_2$.

Let $X, Y \subseteq V(G)$. We say that $(X,Y)$ is a packing pair of $G$ if:
(1) $X$ and $Y$ are two disjoint $2$-packings;
(2) every vertex of $X \cup Y$ belongs to a triangle of $G$; and
(3) each triangle of $G$ contains at most one vertex from $X \cup Y$.
For a packing pair $(X,Y)$, set
$w(X,Y) = \sum_{x \in X \cup Y} w(x), \quad \text{and let } \gamma(X,Y)$ denotes the number of triangles in $G[V(G) \setminus (X \cup Y)]$.
A packing pair $(X,Y)$ is said to be maximum if $w(X,Y)$ is maximum among all packing pairs, and among all such pairs, we choose a maximum packing pair $(A,B)$ minimizing $\gamma(A,B)$.

We claim that $\gamma(A,B) = 0$. Suppose otherwise, and let $T = abc$ be a triangle in $G' = G[V(G) \setminus (A \cup B)]$. 
For each vertex $x \in \{a,b,c\}$ there exists a vertex $u \in A$ with $\operatorname{dist}(x,u) < 3$, since otherwise 
$A' = A \cup \{x\}$ yields a packing pair $(A',B)$ with $w(A',B) > w(A,B)$, contradicting the maximality of $(A,B)$. 
A symmetric statement holds for $B$. \hfill$(*)$\\
We will treat two cases:

\medskip
\noindent\textbf{Case 1.} None of the vertices of $T$ lies in two triangles.

Note that if $u \in \{a,b,c\}$ has a neighbor $x \in A \cup B$, then $x$ is the unique vertex in $A \cup B$ such that $dist(x,u) \le 2$ and the shortest path joining $x$ to $u$ does not intersect $T \setminus \{u\}$. This observation is due to the fact that $x$ lies in a triangle  that does not intersect $(A \cup B) \setminus \{x\}$, according to the definition of a packing pair. Thus, since each vertex of $T$ is within distance two of some vertex in $A$ (resp., in $B$) by $(*)$, 
and since each triangle of $G$ has at most one vertex in $A \cup B$, there exist $X \in \{A,B\}$ and $x \in X$ such that  every vertex of $T$ is at distance at least $3$ from every vertex of $X \setminus \{x\}$.  Consequently, $x$ should be adjacent to a vertex of $T$. Without loss of generality, assume $X = A$ and that $x$ is adjacent to $a$.
By definition of a packing pair, $x$ lies in a triangle $T' = xyz$.
Because none of the vertices of $T$ lies in two triangles, we have $\{a,b,c\} \cap \{x,y,z\} = \emptyset$. At least one of $y$ or $z$ must be adjacent to a vertex of $B$; otherwise, by letting $A' = (A \setminus \{x\}) \cup \{a\}$ and $B' = B \cup \{x\}$, we obtain a packing pair $(A',B')$ with $w(A',B') > w(A,B)$, a contradiction. Assume without loss of generality that $y$ is adjacent to $x' \in B$. Then $x'$ lies in a triangle $T''$. If $T'' = x'yz$, then clearly $x'$ is not contained in two triangles, and hence $w(y)>w(x')$. Thus,
$(A',B')$ with $A' = (A \setminus \{x\}) \cup \{a\}$ and $B' = (B \setminus \{x'\}) \cup \{y\}$ yields $w(A',B') = w(A,B) - k_2 + k_1 + k_2 > w(A,B)$ and $(A',B')$ is a packing pair, a contradiction. Hence, $V(T'') \cap V(T') = \emptyset$. Therefore, $y$ is at distance at least three from each vertex of $A \setminus \{x\}$ since $T''$ contains only one vertex in $A\cup B$. Setting $A' = (A \setminus \{x\}) \cup \{b,y\}$ gives $w(A',B) = w(A,B) - k_2 + 2k_2 > w(A,B)$ with $(A',B)$ is a packing pair, again a contradiction. We will treat two cases:
\medskip

\noindent\textbf{Case 2.} A vertex of $T$ lies in two triangles.

Assume $a$ belongs to a second triangle $T' \neq T$, and without loss of generality $b$ is also a vertex of $T'$, with $x$ the third vertex of $T'$.
Then $x$ is not contained in two triangles and $x \in A \cup B$ by $(*)$; assume $x \in A$. If $c$ has no neighbor in $A$, then the packing pair $(A',B)$ with 
$A' = (A \setminus \{x\}) \cup \{a\}$ satisfies $w(A',B) = w(A,B) - k_2 + k_1 > w(A,B)$, a contradiction. Hence $c$ has a neighbor $c' \in A$.
By the definition of packing pair, $c'$ lies in a triangle $T''$. Clearly, $T''$ is disjoint from $T$.
Since $c'$ is the unique vertex of $T''$ in $A \cup B$, $c$ is at distance at least $3$ from each vertex of $B$, again a contradiction.

\medskip
In all cases we reach a contradiction; thus $\gamma(A,B) = 0$, and the lemma follows.\qed
\end{proof}

\section{New Proof}

We now outline the idea behind our proof.  
Let $(A,B)$ be the packing pair provided by Lemma~2.1, so that $G[V(G)\setminus (A\cup B)]$
contains no triangles. If $G[V(G)\setminus (A\cup B)]$ contains no odd cycle, then it is bipartite and $G$
becomes $(1,1,2,2)$-packing colorable.  
To capture the process of eliminating these odd cycles, we introduce the notion of an odd-cycle 
reduction pair. Let $A',\; B' \subseteq V(G)$. We call $(A',B')$ an odd-cycle reduction pair of $(A,B)$ if: (1) $A'$ and $B'$ are two disjoint 2-packings, (2)
$A\subseteq A'$ and $B\subseteq B'$, (3) every vertex in $A'\setminus A$ or $B'\setminus B$ lies on an odd 
cycle of $G[V(G)\setminus (A\cup B)]$, and (4) every such odd cycle contains at most one vertex of $A'\cup B'$.  
For  an odd-cycle reduction pair $(A',B')$ of $(A,B)$, we define $\theta(A',B')$ to be the number of odd cycles in 
$G[V(G)\setminus (A'\cup B')]$.  
If some reduction pair satisfies $\theta(A',B')=0$, then $G$ is $(1,1,2,2)$-packing colorable.  
While it is unknown whether this occurs in all cubic graphs, we show that it does for claw-free cubic
graphs.

\begin{theorem}
Every claw-free cubic graph is $(1,1,2,2)$-packing colorable.
\end{theorem}
\begin{proof}
Let $(A,B)$ be the packing pair obtained from Lemma~2.1, so that $G[V(G)\setminus (A\cup B)]$ has no triangles.  Among all odd-cycle reduction pairs of $(A,B)$, choose one, denoted $(A^*,B^*)$, such that  $\theta(A^*,B^*)$ is minimum.  We claim that $\theta(A^*,B^*) = 0$. Suppose to the contrary that 
$G^* = G[V(G)\setminus (A^*\cup B^*)]$ contains an odd cycle, say $C = x_1 x_2 \ldots x_{2k+1}$.  

Note that each vertex $x$ of $C$ is at distance less than three from some vertex in $A^*$, since otherwise  $(A^*\cup \{x\}, B^*)$ is an odd-cycle reduction-pair, contradicting the minimality of $\theta$. A similar statement holds for $B^*$.   
Since all triangles have been destroyed in $G^*$ and since $G$ is claw-free, each $x_i$ must be adjacent to a vertex of $A^*\cup B^*$.

A vertex of $C$ cannot lie in two triangles of $G$. Suppose, to the contrary, that $x_i$ lies in two triangles; assume $i=1$. Then $x_{2k+1}$, $x_1$, and $x_2$ share a common neighbor in $A^*\cup B^*$, say $u$. The two triangles containing $x_1$ are then: $x_{2k+1}x_1u$ and $x_1x_2u$. Assume $u\in A^*$, then $x_1$ is at distance at least three from all vertices in $B^*$, contradiction.  Thus each $x_i$ lies in exactly one triangle of $G$.
Consequently, $x_1$ and $x_2$ (or $x_1$ and $x_{2k+1}$) must share a common neighbor in $A^*\cup B^*$. Without loss of generality, assume that $x_1$ and $x_2$ share a neighbor.  Then $x_{2i-1}$ and $x_{2i}$ have a common neighbor in $A^*\cup B^*$ for every $i$, $2 \le i \le k$.
Consequently, $G[\{x_{2k+1}\} \cup N(x_{2k+1})]$ is a claw, a contradiction. 
Therefore, $\theta(A^*,B^*) = 0$, and the result follows by coloring the vertices of $G^*$ by $1_a$ and $1_b$, the vertices of $A^*$ by $2_a$, and the vertices of $B^*$ by $2_b$.
\end{proof}

\section{Open Problem}

The lemma established in this paper shows that triangles can be completely removed from any cubic graph by deleting two disjoint $2$-packings, and the subsequent theorem proves that every claw-free cubic graph is $(1,1,2,2)$-packing colorable. A natural question is whether the claw-free condition can be relaxed.

\medskip
\noindent\textbf{Problem.}
Let $G$ be a cubic graph in which every vertex belongs to a cycle of length $3$ or $4$. Is $G$ necessarily $(1,1,2,2)$-packing colorable?
\medskip

Beyond this, our methods suggest that claw-free cubic graphs might admit an even stronger packing coloring.

\medskip
\noindent\textbf{Problem 2.}
Is every claw-free cubic graph $(1,1,2,3)$-packing colorable?

\section*{Statements and Declarations}
\textbf{Funding:} This work was partially supported by a grant from the IMU-CDC and Simons Foundation.\\
\textbf{Data Availability:} This work describes mathematical results and is not accompanied by original research data.\\
\textbf{Author Contributions:} All authors contributed to the conception of the study, the proofs, and the writing of the manuscript. All authors read and approved the final manuscript.\\
\textbf{Conflict of Interest:} The authors declare that they have no conflict of interest.


\begin{thebibliography}{99}

\bibitem{bkl} J. Balogh, A. Kostochka and X. Liu, Packing chromatic number of cubic graphs, \textit{Discrete Math.} 341(2) (2018), 474--483

\bibitem{3} J. Balogh, A. Kostochka and X. Liu, Packing Chromatic Number of Subdivisions of Cubic Graphs, \textit{Graphs Combin.} 35(2) (2019), 513--537.

\bibitem{bf} B. Bre\v{s}ar and J. Ferme, An infinite family of subcubic graphs with unbounded packing chromatic number,
\textit{Discrete Math.} 341(8) (2018), 2337--2342.

\bibitem{bfk} B. Bre\v{s}ar, J. Ferme, S. Klavžar, D.F. Rall, A survey on packing colorings, \textit{Discuss. Math. Graph Theory} 40 4 (2020), 923-970.
 
\bibitem {11} B. Bre\v{s}ar, N. Gastineau and O. Togni, Packing colorings of subcubic outerplanar graphs, \textit{Aequationes Math.}, 94(5), pages 945-967 2020. 
 
\bibitem{8} B. Bre\v{s}ar, S. Klav\v{z}ar, D.F. Rall and K. Wash, Packing chromatic number under local changes in a graph, \textit{Discrete Math.} 340 (2017), 1110--1115.
  
\bibitem{9} B. Bre\v{s}ar, S. Klav\v{z}ar, D.F. Rall and K.Wash, Packing chromatic number, $(1,1,2,2)$-colorings, and characterizing the Petersen graph, \textit{Aequationes Math.} 91 (2017), 169--184.

\bibitem {BKR} B. Brešar, K. Kuenzel and D. F. Rall, Claw-free cubic graphs are $(1, 1, 2, 2)$-colorable, \textit{Discrete Math.} 348 (8) (2025), 114477.
\bibitem{AM1} A. El Zein and M. Mortada, Advances on the Packing Coloring Conjectures of Subcubic Graphs, \href{https://arxiv.org/abs/2503.20239}.
\bibitem{AM2} A. El Zein and M. Mortada, A Step Toward an $S$-Packing Coloring Conjecture for Subcubic Graphs, submitted.
\bibitem {16} N. Gastineau and O. Togni, $S$-packing colorings of cubic graphs, \textit{Discrete Math.} 339 (2016), 2461--2470.

\bibitem{d} W. Goddard, S.M. Hedetniemi, S.T. Hedetniemi, J.M. Harris and D.F. Rall, Broadcast chromatic numbers of graphs, \textit{Ars Combin.} 86 (2008) 33--49. 
 
\bibitem {24} D. Laiche, I. Bouchemakh and \'{E}. Sopena, Packing coloring of some undirected and oriented coronae graphs,\textit{ Discuss. Math. Graph Theory}, 37 (3) (2017), 666--690.
   
\bibitem {m} R. Liu, X. Liu, M. Rolek and G. Yu, Packing $(1,1,2,2)$-coloring of some subcubic graphs, \textit{Discrete Appl. Math.} 283 (2020), 626--630.

\bibitem{LW}  X. Liu and Y. Wang, Partition subcubic planar graphs into independent sets, arXiV 2408.12189 (2024), math.CO.

\bibitem{LZZ24} X. Liu and X. Zhang and Y. Zhang, Every subcubic graph is packing $(1,1,2,2,3)$-colorable, \textit{Discrete Math.} 348 (11) (2025). 114610. 
\bibitem{AM3} M. Mortada and A. El Zein, On the $(1,1,2,3)$-Packing Coloring of Some Subcubic Graphs, to be published soon.
\bibitem {26} M. Mortada and O. Togni, About $S$-Packing Coloring of Subcubic Graphs, \textit{Discrete Math.} 347(5) (2024).

\bibitem {27} M. Mortada and O. Togni, About $S$-Packing Coloring of 2-Saturated Subcubic Graphs, \textit{Australasian J. Combin.} 90 (2) (2024), 155–-167.

\bibitem {mai} M. Mortada, About $S$-Packing Coloring of  3-Irregular Subcubic Graphs, \textit{Discrete Appl. Math.} 359 (2024), 16--18.

\bibitem {mai1} M. Mortada and O. Togni, Further results and questions on $S$-packing coloring of subcubic graphs, \textit{Discrete Math.} 348(4) (2025), 114376.

\bibitem {25} C. Sloper, An eccentric coloring of trees, \textit{Australasian J. Combin.} 29 (2004), 303-321.

\bibitem{i} B. Tarhini and O. Togni, $S$-Packing Coloring of Cubic Halin Graphs, \textit{Discrete Applied Math.} 349(31) (2024), 53--58.

\bibitem{a}  W. Yang and B. Wu, On packing $S$-colorings of subcubic graphs, \textit{Discrete Applied Math.} 334(2023), 1--14.

\end{thebibliography}
\end{document}